\documentclass[12pt,a4paper]{amsart}

\usepackage{amsthm}
\usepackage{amsmath}
\usepackage{graphicx}
\usepackage{graphics}
\usepackage{amsfonts}
\usepackage{amssymb}
\usepackage{mathabx}
\usepackage{cite}
\usepackage{comment}
\usepackage{verbatim} 
\usepackage{enumerate}
\usepackage{xcolor}
\usepackage{hyperref}

\usepackage[lmargin=25mm,rmargin=25mm,top=25mm,bottom=30mm,paperwidth=210mm,paperheight=11in]{geometry} 
\usepackage{soul} 

\newcommand{\xig}[2]{{\xi}^{#1}_{#2}} 
\newcommand{\xip}[2]{\xi_{#2}(#1)} 
\newcommand{\pfl}{{\eta}_t} 
\newcommand{\etap}[2]{\eta_{#2}(#1)} 
\newcommand{\osp}[2]{\xi_{#2}[#1]}

\newcommand{\I}[1]{{\bf 1}{\left\{#1\right\}}} 

\def\R{\mathbb{R}}
\def\N{\mathbb{N}}
\def\L{\mathcal{L}} 

\let\oldsum\sum
\renewcommand{\sum}{\displaystyle\oldsum} 

\newtheorem{theorem}{Theorem}[section]
\newtheorem{lemma}[theorem]{Lemma}

\newtheorem{corollary}[theorem]{Corollary}

\numberwithin{equation}{section}

\title[F-KPP Scaling Limit]{F-KPP Scaling limit and selection principle for a Brunet-Derrida type particle system.
}

\author{Pablo Groisman}
\address{Departamento de Matem\'atica, FCEN, Universidad de Buenos Aires, IMAS-CONICET and NYU-ECNU Institute of Mathematical Sciences at NYU Shanghai} \email{pgroisma@dm.uba.ar}
\author{Matthieu Jonckheere}
\address{Instituto de C\'alculo, FCEN, Universidad de Buenos Aires and IMAS-CONICET}
\email{mjonckhe@dm.uba.ar}
\author{Juli\'an Mart\'inez}
\address{Departamento de Matem\'atica, FIUBA, Universidad de Buenos Aires, Instituto de C\'alculo-CONICET, FCEN}
\email{jfmartinez@fiuba.uba.ar}

\begin{document}

\begin{abstract}
We study a particle system with the following diffusion-branching-selection mechanism. 
Particles perform independent one dimensional Brownian motions and on top of that, at a constant rate, a pair of particles is chosen uniformly at random and both particles adopt the position of the rightmost one among them. 
We show that the cumulative distribution function of the empirical measure converges to a solution of the Fisher-Kolmogorov-Petrovskii-Piskunov (F-KPP) equation and use this fact to prove that the system selects the minimal macroscopic speed as the number of particles goes to infinity.
\end{abstract}

\subjclass[2010]{32491293823489}



\maketitle

%
\section{Introduction}

The F-KPP equation
\begin{equation}
 \label{KPP}
\begin{array}{l}
\displaystyle {\partial_t u}= \frac{1}{2}  \displaystyle {\partial^2_{x} u} + u^2 -u, \quad x \in \R,\,
t>0,\\
\\
\displaystyle u(0,x)= \displaystyle u_0(x), \quad x \in \R.
\end{array}
\end{equation}
was first introduced as a central model of front propagation for reaction-diffusion phenomena.
Both Fisher and Kolmogorov, Petrovskii and Piskunov \cite{Fisher, KPP} proved independently
the existence of an infinite number of traveling wave solutions.
In 1975, McKean established a first link with a microscopic model (a particle system),
showing an exact connection with Branching Brownian motion (BBM) \cite{MK}.
This sprouted a large effort of research focusing on defining corresponding microscopic models
for front propagation both in the physics and mathematics literature.
The seminal papers \cite{BD2,BD} set the basic questions on the influence of microscopic effects on the front propagation properties.
Three types of models were then considered: adding noise to the deterministic equation \cite{BDMM, MMQ, MMQ2}, introducing a cut-off \cite{BenDep, DP} or defining truly microscopic models via interacting particles systems conserving the total number of particles \cite{BDMM2,DR,M,deMasi}.
This allows to study the effects produced by the finite size population on the front propagation. 
We focus here on the third direction of research.

Brunet and Derrida \cite{BD} highlighted that in contrast to the macroscopic equations, which admit infinitely many traveling waves (characterized by their velocities), microscopic models should select a unique velocity.
Defining finite population particle systems, of size $N$, with branching and selection (in discrete time), they showed using heuristic arguments and simulations that the asymptotic speed of a particle system has a deviation of order $(\log N)^{-2}$ from the macroscopic speed. This sheds light on the effects of the finite size population, as the corresponding speed converges much slower than expected. 
This was rigorously proved in \cite{BG} for $N$-Branching Random Walks ($N$-BRW), a system in which particles branch, perform random walks and undergo a selection mechanism.

Several similar particle systems involving diffusive movements and sequential steps of branching and selection have been considered in this context. 
In \cite{DR} the authors consider a model closely related to $N-$BRW and prove that the distribution function of the empirical measure converges, for fixed times, to a free boundary integro-differential equation in the F-KPP class, as well as a speed selection phenomenon. 
The same conclusions were conjectured to hold in \cite{GJ2} for the $N-$Branching Brownian Motion ($N$-BBM), introduced and studied in \cite{M}. These conclusions have been recently settled in \cite{BBP18, deMasi}. More on this macroscopic equation can be found in \cite{BBD17}.

In the present article, we define and study a finite-population particle system in continuous time, following closely the dynamics proposed by Brunet and Derrida in \cite{BD}. 
In our setting $N$ particles diffuse independently according to one dimensional Brownian motion, branch and get selected, with the important feature that particles are paired at each branching event. 
At these times, the particle on the pair with largest position  branches into two and the one (in the pair) with smallest position is eliminated from the system. 
This corresponds to the branching-selection mechanism.

We first prove that the cumulative distribution function of the empirical measure of this system converges to a solution of the F-KPP equation as the number of particles goes to infinity. As a consequence, we are able to prove that the system selects the minimal speed of the macroscopic model as $N$ goes to infinity. 
In order to do that, we first show the existence of a stationary regime for the system seen from its leftmost particle
and the existence of an asymptotic speed for fixed $N$, using classical arguments. Combining this with the hydrodynamic limit previously obtained, we get a lower bound for the speed.
The upper bound can be obtained by constructing our process as a pruning of $N$ independent BBMs. 

Some particular features of our setting, in contrast with previously mentioned works are 
(i) the proofs do not depend on explicit computations but rather on a careful analysis of the system (for instance, the limiting speed is obtained through comparison with F-KPP solutions and not through Legendre transforms), 
(ii) we provide explicit bounds for the decorrelation between particles and hence we can quantify the propagation of chaos, 
(iii) we believe that the analysis is robust enough to be extended to more general diffusions and branching-selection mechanisms (see \cite{GSL}) and 
(iv) the behavior of the system as $N$ goes to infinity is given exactly by the solution of the F-KPP an not by approximations. 

The rest of the paper is organized as follows.
In Section \ref{sec:mod}, we introduce the model and state the main results. 
The proof of propagation of chaos and convergence towards the F-KPP is given in Section \ref{sec:proof1}. 
Finally in Section \ref{sec:proof3}, we prove the existence and properties of the velocities for fixed $N$ and then proceed to prove the {\em selection principle}: as $N$ goes to infinity, the velocities converge to the minimal velocity of the F-KPP equation.

\section{Model and main results}\label{sec:mod} 
We first describe the dynamics in an informal way.
The system starts with $N$ particles located in the real line, each of which performs independently a standard Brownian motion and carries a Poisson clock with rate one. 
At Poissonian times, the particle chooses another particle uniformly among the other ones. If the choosen particle has larger position, it jumps on top of it. Throughout the article we will refer indistinctly to this interaction as a \emph{jump} of the smaller particle of the pair to the position of the larger one or as a \emph{branching} of the larger particle with killing of the smaller one. More precisely $(\xi_t)_{t\ge0}$ is a Markov process in $\R^N$ with generator given by
\begin{equation}
\label{generator}
\L f(\xi)=\frac{1}{2} \sum_{i=1}^{N} \partial^2_{x_i} f(\xi) + \frac{1}{2 (N-1)} \sum_{i=1}^{N} \sum_{j \neq i} 
\left(f(\theta_{ij}(\xi))-f(\xi)\right),
\end{equation}
for $f \in {C}^2_0(\R^N)$, the set of twice differentiable functions with compact support. Here 
\begin{equation}
\label{teta}
\theta_{ij}(\xi)(k)=
\begin{cases}
\max\{\xi(j), \xi(i)\} & \text{ if $k\in \{i,j\}$},\\
\xi(k) & \text{ otherwise}.
\end{cases}
\end{equation}
Namely, given two labels $i,j$, $\theta_{ij}$ replaces the position of the particle with smaller position with the position of the larger one.

We denote by $\xig{}{t}=(\xi_t(1), \ldots, \xi_t(N))$ the process at time $t$ and define the  \emph{empirical measure} associated to $\xig{}{t}$ by
\begin{equation*}
\mu^N_t:= \frac{1}{N} \sum_{i=1}^{N} \delta_{\xi_t(i)}, \qquad t>0.
\end{equation*}
The \emph{cumulative distribution function} of $\mu^N_t$ is given by,
\begin{equation*}
F_N(x,t):= \mu^N_t((-\infty,x])	, \qquad x\in \R,t>0.
\end{equation*}
For simplicity, we omit the dependence of $\mu^N_t$ and $F_N(x,t)$ on ${\xi}_t$ if not necessary.

\bigskip
\noindent
{\bf Notation.}
In what follows, we denote by $\mathcal{P}$ the space of probability measures on $\R$ endowed with the weak topology. For a polish metric space $S$, $D([0,T], S)$ is the space of c\`adl\`ag functions from $[0,T]$ to $S$ endowed with the Skorokhod topology. 
We assume that all our processes are defined in a probability space $(\Omega, \mathcal F, P)$ and denote with $E$, the expectation with respect to $P$.
When needed, we use $P_\mu$ (respectively $P_{\xig{}{0}}$) whenever the initial condition is distributed at random according to the probability measure $\mu$ (respectively $\delta_{\xig{}{0}}$), with the same convention for expectation.
The $\sigma$-algebra generated by the process up to time $t$ is denoted by $\mathcal{F}_t$. 
For simplicity, we write $u_N(x,t):=E_{\xig{}{0}}[F_N(x,t)]$ and $\| \cdot \| = \| \cdot \|_\infty$ for the infinity norm of a real valued function.

\subsection{Hydrodynamic limit}
The following theorem provides a link between the aforementioned model and the solutions of the F-KPP equation,
through the hydrodynamic limit.

\bigskip

\begin{theorem}
\label{mresult} 
Let $u_0$ be a distribution function in $\R$ 
such that $\lim_{N\to\infty} \|F_N(\cdot,0)-u_0\| = 0$ in probability and $u$ the solution to \eqref{KPP}.
Then, for any $t>0$ we have 
\begin{equation}
\label{convergence}
\lim_{N\to\infty}\|F_N(\cdot,t) - u(\cdot,t)\| = 0 \qquad \text{in probability}.
\end{equation}	
\end{theorem}


{We can also state a path-wise version of this result.}

\begin{corollary}
\label{path}
In the setting of Theorem \ref{mresult}, let $\mu_t$ be the deterministic measure in $\R$ given by $\mu_t((-\infty,x])=u(x,t)$. Then, 
\begin{equation*}
\lim_{N\to\infty}(\mu^N_t)_{t\ge 0}  = (\mu_t)_{t \ge 0},
\end{equation*}
weakly in $D([0,+\infty),\mathcal P)$, in probability.
\end{corollary}

\subsection{Speed selection} For $\xi \in \R^N$, denote $\xi[i]$ the $i$-th order statistic of $\xi$, i.e.
\[
 \xi[1]=\min_{1\le j \le N}\xi(j), \qquad \xi[i+1] = \min_{j\colon \xi(j)\ge\xi[i]} \xi(j).
\]
We use the label of the particles to break ties. The position of the $i$-th particle as seen from the minimum is given by $\etap{i}{t}=\osp{i+1}{t}-\osp{1}{t}$ and the process as seen from the minimum is defined by
\begin{equation}
\label{leftmost}
\pfl=(\etap{1}{t},\ldots,\etap{N-1}{t}).
\end{equation}
\begin{theorem}
\label{finiteN}
Let $N \geq 1$. Then,
\begin{enumerate}
\item
The process $(\pfl)_{t \geq 0}$ has a unique stationary distribution $\nu^N$, which is absolutely continuous with respect to the $N$-dimensional Lebesgue measure.
\item 
For any initial distribution $\rho^N_0$ 
$$\lim_{t\to \infty}\| P_{\rho^N_0}(\pfl \in \cdot)-\nu^N(\cdot)\|_{TV}= 0.$$
\item
There exists $v_N \geq 0$ such that
\begin{equation}
\label{velocityE}
\lim_{t \to \infty} \frac{\osp{1}{t}}{t}=\lim_{t \to \infty} \frac{\osp{N}{t}}{t}=v_N, \quad \text{almost surely and in $L^1$.}
\end{equation}
\item $v_{N+1}\ge v_N$ and
\begin{equation}
\label{minimal}
\lim_{N \to \infty} v_N=\sqrt{2}.
\end{equation}
\end{enumerate}
\end{theorem}

{\bf Remark.}
Equation \eqref{minimal} can be thought as a \emph{selection principle} in the sense that though the macroscopic equation admits infinitely many travelling waves with different velocities, for each $N$ and for \emph{any} initial distribution the microscopic system has a \emph{unique} velocity which converges to the \emph{minimal velocity} of the macroscopic equation, when $N$ goes to infinity. 
This kind of result was first proved in \cite{BCdMFLS}. 


\section{Hydrodynamic limit}\label{sec:proof1}
We begin now with the proof of Theorem \ref{mresult}. 
The main idea is to show that $u_N$ verifies, in the limit as $N$ goes to infinity, the F-KPP equation and that the variances $V_{\xig{}{0}}[F_N(x,t)]$ converge to zero, which boils down to prove asymptotic decorrelation of the particles (propagation of chaos).
The following graphical construction will turn to be instrumental for this purpose.

\subsection{Graphical construction}
\label{graphicalconstruction}
We construct the process as a deterministic function of $N$ independent Brownian motions and an homogeneous Poisson process (for the jumps).\\
To each particle $i$, we associate a marked Poisson process $\omega^i$ and an independent standard Brownian motion $(B^i_t)_{t\geq 0}$.
The process $\omega^i$ is defined on $\R \times (\{1,\ldots,N\} \backslash \{i\})$ with intensity measure $dt d\beta_i$, where $\beta_i$ is the uniform distribution on $\{1,\ldots,N\} \backslash \{i\}$. The processes $(\omega^i, (B^i_t)_{t \geq 0})_{1 \leq i \leq N}$ are independent.
Consider the superposition  $\omega=\bigcup \omega^i {\times \{i\}}  $ and sort the marks in order of appearance. 
We denote this sequence with $(\tau_k,j_k,i_k)_{k \in \N}$ where the first coordinate $\tau$ is the {\em time mark}, the second one is the {\em partner mark} chosen according to $\beta_i$ and the third one is the {\em label mark} (the label mark is $i$ if it occurs at $\omega^i$). We define $\xig{}{t}$ inductively as follows.
\begin{itemize}
\item 
At time $0$, the configuration is $ \xig{}{0}$.
\item
Assume $ \xig{}{\tau_k}$ is defined. For $t \in (\tau_k, \tau_{k+1}]$ we define
\begin{equation*}
\xip{l}{t}=
\begin{cases}
\xig{}{t^-}(j_{k+1}) & \text{ if $t=\tau_{k+1}$, $l=i_{k+1}$ and $\xig{}{t^-}(i_{k+1}) < \xig{}{t^-}(j_{k+1})$}\\
\xip{l}{\tau_{k}} + B^l_t - B^l_{\tau_k}  & \text{ otherwise.}
\end{cases}
\end{equation*} 
\end{itemize}

It is straightforward to check that $( \xig{}{t})_{t \geq 0}$, as constructed above is Markov, with generator given by \eqref{generator}.

For each particle $i$ and a time $t>0$, we construct backwards a set of (labels of) particles that we call {\em ancestors}. 
The important feature of this set is that the position of the particle is measurable with respect to the $\sigma$-algebra generated by the Brownian motions and Poisson processes \emph{attached to the ancestors up to time $t$}. 
For any two particles, conditioned on the event that their {\em clan of ancestors} do not intersect, they are independent. 
More details below.

\subsection{The clan of ancestors}
Given a set of labels $A \subseteq \{1,\ldots,N\}$ we denote by $\omega^A[0,s]$ the superposition of all the points of $\omega^j$ with $j \in A$ restricted to the interval $[0,s]$. 
We add $(0,0,0)$ to this set for convenience. That is, $\omega^A[0,s]=\bigcup_{j\in A} \omega^j([0,s]) \bigcup \{(0,0,0)\}$.
For each $1\le i \le N$, we define inductively a sequence of times and sets as follows.
\begin{itemize}
\item Let $s_1$ be the largest time mark in $\omega^{\{i\}}[0,t]$ and $j_1$ its partner mark. If $s_1=0$ stop and define $\psi^i_t=\{i\}$. Otherwise define $A^i_1=\{i,j_1\}$.
\item
Suppose $s_k, A^i_k$ are defined. 
Then, let 
$s_{k+1}$ be the largest time in $\omega^{A^i_k}[0,s_k]$ and $j_{k+1}$ its partner mark. If $s_{k+1}=0$ stop and define $\psi^i_t=A^i_k$. Otherwise define $A^i_{k+1}=A^i_k \cup \{j_{k+1}\}$.
\end{itemize}
As $N$ is finite, this procedure finishes after a finite number of steps  $n^ i_t$  almost surely.

The next lemma gives a bound on the probability of intersection of two clans of ancestors which allows us to control the two-particles correlation.

\begin{lemma}
\label{influence}
For $i \ne j$  and $t>0$,
\begin{equation*}
P_{{{\mu^N_0}}}(\psi^i_t\cap \psi^j_t \neq \varnothing) \leq \frac{e^{t}-1}{N-1}.
\end{equation*}
\end{lemma}

In the event $\{\psi^i_t \cap \psi^j_t= \varnothing\}$ the only possible dependence between $\xi^i_t$ and $\xi^j_t$ is due to the initial condition. As a consequence, for deterministic initial conditions, we obtain the following bound.

\begin{lemma} For $t\ge 0$ and $x,y \in \R$, we have 
\label{Lemma:decorrelation}
\begin{equation*}
\sup_{ \xig{}{} \in \R^N} \left|\sum_{i,j=1}^{N} E_{ \xig{}{}}[\I{\xip{i}{t} \leq x}\I{\xip{j}{t} \leq y}]-E_{ \xig{}{}}[\I{\xip{i}{t}\leq x}]E_{\xig{}{}}[\I{\xip{j}{t}\leq y}] \right| \leq 2N e^{t}.
\end{equation*}
\end{lemma}
The proof of this lemma is similar to the one in \cite[Section 3]{AFG} once we have Lemma \ref{influence}.
The only minor modification is to consider the events $\{\xip{i}{t} \leq x, \xip{j}{t} \leq y\}$ instead of $\{\xip{i}{t}=x, \xip{j}{t}=y\}$.
For the proof of Lemma \ref{influence} we also refer the reader to \cite[Lemma 2.1]{AFG}.

The following lemma allows to control perturbations of the F-KPP equation.

\begin{lemma}
\label{COROCP}
Let $u$ be a solution of \eqref{KPP} with initial datum $u_0$ and $w$ a smooth bounded function in $\R\times [0,T]$ such that
\[
  \partial_t w = \frac12\partial^2_{x} w + w^2 - w + g, \qquad w(x,0)=w_0(x).
\]
Then
\begin{equation*}
 \| u(\cdot, t) - w(\cdot, t)\| \le \left (\|u_0 - w_0\| + \int_0^t e^{-s} \|g(\cdot, s)\| \, ds\right ) e^{{C}t} ,
\end{equation*}
{with $C=\max\{\|g\|T + \|w_0\|,1\}$.}
\end{lemma}

\begin{proof}
Call $\underline z=w-u$ and 
\[
\overline z(x,t)= \left (\|u_0 - w_0\| + \int_0^t e^{-s} \|g(\cdot, s)\| \, ds\right ) e^{Ct}.
\]
By the maximum principle for parabolic equations in unbounded domains \cite[Proposition 52.4]{QS} we get $\|w\| \le C$ and hence, this function verifies
\begin{align*}
\partial_t \overline z =&\, \frac12\partial^2_{x} \overline z + C\overline z + \|g(\cdot,t)\|\\ 	
\ge & \, \frac12\partial^2_{x} \overline z + (w+u-1) \overline z + \|g(\cdot,t)\|,
\end{align*}
with $\overline z(x,0)=\|u_0 - w_0\|$, while for $\underline z$ we have
\begin{align*}
\partial_t \underline z =&\, \frac12\partial^2_{x} \underline z + w^2-u^2 -  \underline z + g\\ 
= & \, \frac12\partial^2_{x} \underline z + (w+u-1) \underline z + g.
\end{align*}
Hence $e=\underline z - \overline z$ verifies 
\[
\partial_t e  \le \frac12\partial^2_{x} e + (w+u-1)e.\\
\]
With $e(x,0)=w_0(x)-u_0(x) - \|u_0 - w_0\| \le 0$. Since $u$ and $w$ are bounded, using again the maximum principle we obtain $e(x,t)\le 0$ for all $(x,t) \in \R\times[0,T]$. Proceeding in the same way with $-\overline z - \underline z$ we get that for all $(x,t) \in \R\times[0,T]$, $-\overline z(x,t) \le  w(x,t) - u(x,t) \le \overline z(x,t)$, which conludes the proof.
\end{proof}

We are ready to prove Theorem \ref{mresult}.
 
\begin{proof}[Proof of Theorem \ref{mresult}]
Throughout this proof, we denote $E_{\xig{}{0}}, V_{ \xig{}{0}}$ and $P_{ \xig{}{0}}$ briefly by $E,V$ and $P$, respectively. We also use $\mathcal{G}(x,t;z)=\int_{-\infty}^x (2\pi t)^{-1/2}e^{-(y-z)^2/2t}\, dy$ for the probability that a Brownian motion started at $z$ is less than $x$ at time $t$.
The proof is divided into two steps.

\bigskip
\noindent
{\em Step 1.} Derivation of the equation for $u_N(x,t)$.

\medskip
Let us first observe that 
\begin{equation}
\label{EP}
u_N(x,t)=\frac{ 1}{N} \sum_{i=1}^N P(\xip{i}{t} \leq x).
\end{equation}
We will focus on each term of the previous sum separately. 
Fix $i$ and let $\tau$ be the {last time-mark before time $t$ of $\omega^i$}.
If there is no such a mark, we set $\tau=0$.
Notice that particle $i$ performs a Brownian motion in $(\tau,t)$ and that $\tau$ is the minimum between $t$ and an exponential random variable with mean $1$.
Thus, by conditioning on $\tau$ we get
\begin{align*}
P(\xip{i}{t}\leq x)
&= e^{-t} \mathcal{G}(x,t;\xip{i}{0}) + 
\int_{0}^{t} e^{-s} E[\mathcal{G}(x,t-s;\xip{i}{s})| \tau=s] ds.
\end{align*}
Observe that for any $z$, $\mathcal{G}(\cdot, \cdot;z)$ verifies the heat equation. 
Hence, for every $1\le i \le N$, the function $q_i(x,t):=P(\xip{i}{t}\leq x)$ verifies,
\begin{equation}
\label{Efep}
\partial_t q_i= \frac12\partial^2_{x} q_i - q_i + P(\xip{i}{t}\leq x |\tau=t),
\end{equation}
We proceed to examine the last term in \eqref{Efep}. 
Let $J_{ij}$ be the event that the partner mark is $j$.
Conditioning on $J_{ij}$ we get
\begin{align}
\nonumber
P(\xip{i}{t}\leq x |\tau=t)&=
\sum_{j \neq i} P \left(\max(\xip{i}{t^-},\xip{j}{t^-})\leq x |\tau=t,J_{ij} \right) P(J_{ij}|\tau=t)\\
\nonumber &= \frac{1}{N-1} E \left ( \sum_{j \neq i} \I{\xip{i}{t}\leq x}\I{\xip{j}{t}\leq x} \right )\\
\label{vt}
&=\frac{1}{N-1} E \left ( \sum_{j=1}^N \I{\xip{i}{t}\leq x}\I{\xip{j}{t}\leq x} - \I{\xip{i}{t}\leq x}\right ).
\end{align}
Since $F_N^2(x,t)=\frac{1}{N^2} \sum_{i,j=1}^{N} \I{\xip{i}{t}\leq x}\I{\xip{j}{t}\leq x}$, combining \eqref{EP},\eqref{Efep} and \eqref{vt} we get 
\begin{align}
\nonumber
\partial_t u_N &= \partial^2_{x} u_N - \tfrac{N}{N-1} E[F_N (1-F_N)]\\
\label{KPPwe}
&= \partial^2_{x} u_N - u_N (1-u_N) + \tfrac{N}{N-1} V(F_N(x,t)) + \tfrac{1}{N-1} (u_N (1-u_N)).
\end{align}

Next we prove that solutions of  \eqref{KPPwe} converge to solutions of \eqref{KPP} when $N$ tends to infinity.

\bigskip

\noindent
{\em Step 2.} Control of the variance and stability of the F-KPP.

\medskip
For $\varepsilon>0$ we have,
\begin{align}
\label{cheby}
\nonumber
P_{\rho^N_0}(|F_N(x,t)& -u(x,t)|> \varepsilon) 
= \int P_{\xig{}{0}} (|F_N(x,t)-u(x,t)|> \varepsilon) \ \rho^N_0(d \xig{}{0})\\
& \nonumber
\leq \int P_{ \xig{}{0}}(|F_N(x,t)-u_N(x,t)|> \tfrac{\varepsilon}{2}) +
P_{ \xig{}{0}}(|u_N(x,t)-u(x,t)|> \tfrac{\varepsilon}{2}) \ \rho^N_0(d \xig{}{0})\\
&\leq \int \frac{4}{\varepsilon^2} V_{\xig{}{0}}(F_N(x,t)) + P_{\xig{}{0}}(\|u_N(\cdot,t)-u(\cdot,t)\|> \tfrac{\varepsilon}{2}) \ \rho^N_0(d \xig{}{0}).
\end{align}
By Lemma \ref{Lemma:decorrelation}, we get that
\begin{equation}
\label{v-d}
\sup_{ \xig{}{0}} V_{ \xig{}{0}}(F_N(x,t)) \leq \tfrac{2}{N} e^T.
\end{equation}
For the second term in the integral in \eqref{cheby} we apply {Lemma} \ref{COROCP} with $w=u_N$ and $g=\tfrac{N}{N-1} V(F_N(x,t)) + \tfrac{1}{N-1} (u_N (1-u_N))$ to obtain
\begin{align*}
\nonumber
\|u_N(\cdot,t)-u(\cdot,t)\| &\leq \left(\|F_N(\cdot, 0) - u_0(\cdot)\| + \int_0^t e^{-s} \|g(\cdot, s)\| \, ds\right) e^{Ct}\\
& \leq \left(\|F_N(\cdot, 0) - u_0(\cdot)\| + \tfrac{1}{N-1} (1+ 2 e^{T}) \right ) e^{CT}.
\end{align*}
In the last inequality we use that $0 \leq u_N \leq 1$ and \eqref{v-d}.
Taking $N$ such that ${C<2}$ and $\tfrac{1}{N-1} (1+ 2 e^{T}){e^{2T}}<\tfrac{\varepsilon}{4}$ and using the two previous observations we can bound \eqref{cheby} by
\begin{equation}
\label{finalcheby}
\frac{8}{N \varepsilon^2} e^T + \rho_0^N(\| F_N(\cdot, 0) - u_0(\cdot) \|> \tfrac{\varepsilon}{4} e^{-CT}),
\end{equation}
and hence, due to our assumption on $\rho^N_0$, for every $x\in \R$ and $t>0$, $F_N(x,t) \to u(x,t)$ in probability. Since $u(\cdot,t)$ is a continuous distribution function, the convergence holds uniformly \cite[Remark 5.28]{GeorgiiSto}, which proves \eqref{convergence}.

Note also that the upper bound \eqref{finalcheby} is independent of $x$ and $t$, hence the convergence in probability holds uniformly in $[0,T]$.
\end{proof}

{\em Proof of Corollary \ref{path}.}
The main ingredient to obtain Corollary \ref{path} is to show that for every $T>0$, the sequence $(\mu^N_t)_{0\le t\le T}$ is tight in $D([0,T],\mathcal P)$. 
Afterwards, Theorem \ref{mresult} implies that there is a unique limit point and hence the desired result.
\bigskip

\noindent
{\em Tightness.} \ 
To prove tightness of $(\mu^N_t)_{0 \leq t \leq T}$ it is enough to show that for any continuous and bounded function $\varphi$ on $\R$ the sequence of real-valued processes $(\langle \mu^N_t,\varphi\rangle)_{0 \leq t \leq T}$, with  $\langle \mu^N_t,\varphi\rangle=\int_{-\infty}^{+\infty} \varphi(x) \mu^N_t(dx)$, is tight in $D([0,T], \R)$ \cite[Theorem 16.16]{kallenberg}.

According to Aldous criterion \cite[Theorem 16.11  and Lemma 16.12]{kallenberg} the following two conditions are enough to get tightness for the previous sequence:
\begin{enumerate}
\item
For every $t \in [0,T]\cap\mathbb{Q}$ and every $\varepsilon>0$, there is an $L>0$ such that
$$\sup_{N>0} P_{}(|\langle \mu^N_t, \varphi \rangle|>L)\leq \varepsilon.$$
\item
Let $\mathfrak T^N_T$ be the collection of stopping times with respect to the natural filtration associated to $\langle\mu^N_t,\varphi\rangle$ that are almost surely bounded by $T$. For every $\varepsilon>0$
$$\lim_{r \to 0} \limsup_{N \to \infty}\sup_{\stackrel{\kappa \in \mathfrak T^N_T}{s<r}} P_{}\Big(|\langle \mu^N_{(\kappa+s)\wedge T}, \varphi \rangle - \langle \mu^N_\kappa, \varphi \rangle|>\varepsilon \Big)=0.$$
\end{enumerate}

The first condition follows immediately by taking $L> \|\varphi\|$. For the second one, we decompose $\varphi$ as the sum of two functions, with one of them uniformly continuous.
Let $\chi_x$ be the indicator function of the compact interval $[-x,x]$. 
We have that
$$\langle \mu^N_t, \varphi \rangle= \langle \mu^N_t, \varphi \chi_x \rangle + \langle \mu^N_t, \varphi (1-\chi_x) \rangle.$$
We hence obtain
\begin{equation}
\label{ineq:comp}
\begin{split}
P(|\langle \mu^N_{(\kappa+s)\wedge T}, \varphi  \rangle - \langle \mu^N_\kappa, \varphi \rangle|>\varepsilon) &\le P_{}(|\langle \mu^N_{(\kappa+s)\wedge T}, \varphi \chi_x  \rangle - \langle \mu^N_\kappa, \varphi  \chi_x \rangle|>\varepsilon/2)\\
& + P_{}(|\langle \mu^N_{(\kappa+s)\wedge T}, (1-\chi_x)  \rangle - \langle \mu^N_\kappa, (1-\chi_x) \rangle|>\varepsilon/(2 \|\varphi\| )).
\end{split}
\end{equation}

We first deal with the first term of the r.h.s. of the previous inequality.
Fix $N>0, \varepsilon>0$ and $\kappa \in \mathfrak T^N_T$ and let $\mathcal{N}^i$ be the total number of marks of $\omega^i[0,T]$, with $(\tau^i_k)_{1 \leq k \leq \mathcal{N}^i}$ its time-marks.

We abbreviate $\varphi \chi_x$ to $\psi$. 
Using the uniform continuity of $\psi$ in a compact interval, we get that there exists $a>0$ such that
$$\sup_{\stackrel{z,z' \in [-x,x]}{|z- z'| \le a}} |\psi(z)-\psi(z')| \le \varepsilon/12.$$
We thus have that 
$$|\psi(z)-\psi(z')| \leq 2 \|\psi\| \I{|z-z'|>a} + \varepsilon/12.$$

We {now} rewrite $|\langle \mu^N_{(\kappa+s) \wedge T}, \psi \rangle - \langle \mu^N_\kappa, \psi \rangle|$ as
\begin{equation*}
\begin{split}
\Bigg|
\tfrac{1}{N} \sum_{i=1}^{N} \Bigg[
& \psi(\xip{i}{(\kappa+s)\wedge T})-\psi(\xip{i}{\tau^i_{\mathcal{N}^i}}) + \psi(\xip{i}{{\tau^{i}_1}-})-\psi(\xip{i}{\kappa}) \\
&  + \sum_{k=2}^{\mathcal{N}^i} \psi(\xip{i}{{\tau^i_{k}}-}) - \psi(\xip{i}{\tau^{i}_{k-1}})
+ \sum_{k=1}^{\mathcal{N}^i} \psi(\xip{i}{\tau^{i}_{k}}) - \psi(\xip{i}{{\tau^i_{k}}-}) \Bigg]
\Bigg|,
\end{split}
\end{equation*}
which can be bounded above by
\begin{equation*}
\begin{split}
\tfrac{1}{N} \sum_{i=1}^{N} \bigg[ 2 \|\psi\| \bigg(
\I{|\xip{i}{(\kappa+s)\wedge T}-\xip{i}{\tau^i_{\mathcal{N}^i}}|>a} 
+ 
\I{|\xip{i}{{\tau^{i}_1}-}-\xip{i}{\kappa}|>a} \bigg) + \frac{\varepsilon}{6} \bigg]
+
\frac{4 \|\psi\| }{N} \mathcal{N}_s,
\end{split}
\end{equation*}
being $\mathcal{N}_s$ the total number of jumps taking place in the interval $[\kappa, \kappa +s]$,
Then,
\begin{equation*}
\begin{split}
P_{}(|\langle \mu^N_{(\kappa+s) \wedge T},& \psi \rangle - \langle \mu^N_\kappa, \psi \rangle|> \tfrac{\varepsilon}{2})  \leq
P_{}(\tfrac{4}{N} \|\psi\| \mathcal{N}_s > \tfrac{\varepsilon}{6}) \\
& + P_{}\bigg(
\tfrac{2 \|\psi\|}{N} \sum_{i=1}^{N} 
\I{|\xip{i}{(\kappa+s)\wedge T}-\xip{i}{\tau^i_{\mathcal{N}^i}}|>a} 
+ 
\I{|\xip{i}{{\tau^{i}_1}-}-\xip{i}{\kappa}|>a} > \tfrac{\varepsilon}{6}
\bigg)
\\
& \leq \frac{12}{N \varepsilon} \|\psi\| E(\mathcal{N}_s) 
+ 
\frac{12}{\varepsilon} \|\psi\| \sup_{0 \leq r \leq s} P (|B_r|>a) \\
& \leq \frac{12}{\varepsilon} s \|\psi\| + \frac{12}{\varepsilon} \|\psi\|  (2-2 \Phi(\tfrac{a}{\sqrt{s}})),
\end{split}
\end{equation*} 
where $\Phi$ is the standard Gaussian distribution function. 
In the third line we applied Markov's inequality and use the fact that the displacement of the particles before (after) the first (last) jump is given by independent Brownian motions. Then we use that $\mathcal{N}$ is a Poisson process.

The second term in (\ref{ineq:comp}) can be controlled using {\em Step 2} above and the fact that the solution of the F-KPP equation is uniformly continuous in compact space-time intervals. Hence the second condition is fulfilled as well.

\medskip

\noindent
{\em Uniqueness of limit points.} 
Let $(\mu^{N_k}_t)_{0 \leq t \leq T}$ be any convergent subsequence and $(\mu_t)_{0 \leq t \leq T}$ its limit. 
Its time marginals $\mu_t$ are characterized by their distribution functions $F_{\mu_t}$. 
Due to Theorem \ref{mresult} we have that $F_{\mu_t}=u(\cdot, t)$. 
This, together with \cite[Theorem 13.1]{Bi99} provide us the desired conclusion.

\section{Speed selection}\label{sec:proof3}
In this section, we prove Theorem \ref{finiteN}. 
The proofs for fixed $N$ are similar to the ones in \cite{DR,BG}. 
We include them for the reader's convenience but we point to those references for the details. 
However, the proof of convergence of the velocities $v_N \nearrow \sqrt{2}$ as $N\to \infty$ requires a completely different strategy and is based on the hydrodynamic limit, Theorem \ref{mresult}.

Parts {\em (1)} and {\em (2)} of Theorem \ref{finiteN} are obtained by proving positive Harris recurrence for the process $(\pfl)_{t \geq 0}$ and part {\em (3)} by means of the subadditive ergodic theorem. 
The monotonicity of the velocities {\em (4)}, is proved by coupling two process with different number of particles but comparable initial conditions. With some abuse of notation we will denote $(\eta_k)_{k \geq 0}$ a process constructed in this section which coincides in law with $(\eta_t)_{t\ge0}$ observed at jump times, although the construction is different.


In what follows we omit the dependence of some of the variables on $N$. 
Consider a Poisson process $(T_k)_{k \geq 1}$ with rate {$\frac{N}{2}$} and an i.i.d. family $(V_k)_{k \geq 1}$ with $V_k=(V_k^1,V_k^2)\sim \mathcal{U}\{O_N\}$, where $O_N=\{(i,j) \colon \ 1\leq i<j \leq N \}$.
Let $({B}^N_s)_{s\geq 0}$ be an $N-$dimensional standard Brownian motion.
For $\xig{}{} \in \R^N$ we denote $\sigma(\xig{}{})=(\xig{}{}[1],\ldots, \xig{}{}[N])$.

We construct a discrete time auxiliary Markov process ${\zeta}_k$ as follows. Given $\zeta_0$, define
\begin{equation}
\label{coupling}
{\zeta}_{k+1}:= \sigma\big(\theta_{{V}_k}(\sigma({\zeta}_{k} + {B}^N_{T_{k+1}}-{B}^N_{T_k} ))\big) 
\end{equation}
with $\theta_{ij}$ defined in \eqref{teta} and assuming $T_0=0$. In words, between jump times, particles evolve according to an $N-$dimensional Brownian motion.
At jump times, two ranks $V=(V^1,V^2)$ are chosen uniformly in $O_N$ and the particle with rank $V^1$ jumps to the position of the particle with rank $V^2$. 
Afterwards, particles are sorted in increasing order. We omit the trajectories between jumps, so that $\zeta_k$ is the state of the system at the $k-$th jump of $(\xi_t)$.

The law of $({\zeta}_k)_{k \geq 0}$ is given by the order statistics of a process with generator \eqref{generator} at jump times as can be seen by direct computation.
This process is a deterministic function of the initial condition $\zeta_0$, the displacements $D_{k}:={B}^N_{T_{k}}-{B}^N_{T_{k-1}}$ and the jump marks $(V_k)_{k \geq 1}$. We emphasize this by writing
\begin{equation}
\label{det.function}
 (\zeta_k)_{k\ge 0}= \Upsilon((V_k)_{k\ge 1},(D_k)_{k\ge 1},\zeta_0).
\end{equation}

As in \eqref{leftmost}, let $(\eta_k)_{k \geq 0}$ be the process $(\zeta_{k})_{k \geq 0}$ as seen from the leftmost particle and 
$\Omega_N:=\{(z_1, \ldots,z_{N-1}) \in \R^{N-1} : 0 \leq z_1 \leq \ldots \leq  z_{N-1} \}$ its state space.
Any set $R \subseteq \Omega_N$ is embedded in $\R^N$ by defining $R^0=\{0\}\times R$. We now turn to prove each of the statements stated in Theorem \ref{finiteN}.

\subsection{Positive Harris recurrence.} 
To prove Positive Harris recurrence it is sufficient (see for instance \cite{asmussen2003}) to show that there exist a set $R\subseteq \Omega_N$ such that
\begin{enumerate}[(i)]
\item 
\label{C1}
$E_\eta(\tau_R) < \infty$, for all {$\eta \in \Omega_N$}, where $\tau_R=\inf\{k\geq0 \ : \ \eta_k \in R\} $. 
\item
\label{C2}
There exists a probability measure $q$ on $R$, $\lambda>0$ and $r \in \N$ so that \\
$P_{\eta_0}(\eta_r \in \tilde{R}) \geq \lambda q(\tilde{R})$ for all {$\eta_0 \in R$} and all $\tilde{R} \subseteq R$.

\end{enumerate}
Take
$$R=\{ \eta \in \Omega_N \ : \ \etap{i+1}{}-\etap{i}{} \in (0,N) \text{ \ for \ } i=1, \ldots, N-1\}.$$
We claim that any initial condition $\eta \in \Omega_N$ can be driven by the dynamics into the set $R$ in exactly {$N-1$} steps with a positive probability uniformly bounded below away from zero. In fact that is the case if (a) each particle does not displace more than ${1}/{2}$ between jumps (due to Brownian movement) and
(b) the particles chosen by $(V_k)_{k \geq 1}$ to perform the jumps are always the first and the last one (hence the first particle jumps to the position of the last one). In this event we have that $\eta_{N-1} \in R$.

The probability of this event can be bounded below uniformly on the initial condition by
\begin{equation}
\label{ub}
P_\eta(\eta_{N-1} \in R) \geq {\left ( {a_N}P(V_1=(1,N))  \right )^{N-1}}= \left( \frac{{2} a_N}{N(N-1)} \right )^{N-1},
\end{equation}
with $a_N>0$ being the probability that a Brownian motion started at the origin does not leave $(-1/2,1/2)$ before a time given by an independent exponential random variable with parameter $N-1$.

This implies \eqref{C1}. 
To prove \eqref{C2} note that if {$\tilde{R}\subseteq R$}, then 
$$\tilde{R}=\{\eta \in \Omega_N \ : \ \eta(i+1)-\eta(i) \in \tilde{R}_i \subseteq (0,N), \quad i=1, \ldots, N-1\}.$$
Observe also that the displacement of the particles between jumps has the law of the order statistics for $N$ independent Brownian motions, whose density is uniformly bounded below for any initial condition {$\eta_0 \in R$} by a positive constant $c$.
Taking $q$ as the normalized Lebesgue measure restricted to $R$, $r=1$ and $\lambda=c$ we get that \eqref{C2} holds.
It is easy to check, by means of \eqref{ub} and the strong Markov property, that $\sup_\eta E_\eta(\tau_R) < \infty$ and hence $(\eta_k)_{k \geq 0}$ is positive Harris recurrent, which implies that a unique invariant measure exists and is finite.
The fact that $\nu^N$ is absolutely continuous follows from the fact that the distribution of the process is absolutely continuous for every positive time, for every initial distribution.

\subsection{Convergence to equilibrium.}
To prove {\em (2)} it suffices to show the result for the subsequences $(\eta_{N m + j})_{m \geq 0}$, with $0\leq j < N$.
Furthermore, due to the Markov property, we only need to consider the case $j=0$.
The inequality obtained in \eqref{ub} (which is independent of $\eta$) implies that $\sup_\eta P_\eta(\tau_R>t)<1$ for some $t>0$. 
The result follows by applying Theorem $4.1 (ii)$ in \cite{AN78} to the subsequence $(\eta_{N m})_{m \geq 0}$.

\medskip

\subsection{Existence of a velocity.} We prove that both limits in \eqref{velocityE} exist almost surely and in $L^1$ and that they are nonrandom as in \cite{BG}.
The argument makes use of the subadditive ergodic theorem \cite{Durrett1991}.
First observe that due to the monotonicity property of the coupling introduced in \eqref{coupling} and translation invariance, it is sufficient to prove the result when all the particles start at the origin. \\
We construct several copies of our process simultaneously. 
The index $n$ refers to the $n$-th copy and the index $k$ refers to time.
Given the sequences $(D_k)_{k \geq 1}$ for the displacements and $(V_k)_{k \geq 1}$ for the jumps we construct each copy of the process in the following manner:
\begin{equation*}
\zeta_{n,0}=\xi_0, \qquad (\zeta_{n,k})_{k\geq 1}=\Upsilon((V_{r})_{r \geq n}, (D_{r})_{r \geq n},\zeta_0), \qquad n\geq 0.
\end{equation*}
Here $\Upsilon$ refers to the construction introduced in \eqref{det.function}.
Let us stress the fact that the $n$-th copy of the process does not make use of the first $n$ elements of the sequences $(V_{r})_{r \geq 1}, (D_{r})_{r \geq 1}$.
Thus, for $d \in \N$,  $(\zeta_{dm,d})_{m\geq 1}$ is an i.i.d. sequence and the distribution of $(\zeta_{n,k})_{k \geq 1}$ does not depend on $n$. 

A key observation is that if we run the process up to time $k$, next we put all the particles at the position of $\zeta_k[N]$, and finally we run it for another $l$ steps, we obtain a configuration that dominates the one obtained by running the process $k+l$ steps.
In other words, 
\begin{equation*}
\zeta_{0,n+k} [N] \leq \{\zeta_{n,k} + \zeta_{0,n}[N]\}[N].
\end{equation*}
Now, define $\gamma_{n,k}=\zeta_{n,k-n}$ for $0\leq n\leq k$.
Taking into account the aforementioned considerations it is simple to verify that the sequence $(\gamma_{n,k}[N])_{n,k \geq 1}$ fulfills all the conditions of the subadditive ergodic theorem. 
Thus, $\lim_{k\to \infty} {\gamma_{0,k}[N]}/{k}$ exists almost surely and in $L^1$, and it is nonrandom.
Finally, observe that $(\gamma_{0,k})_{k \geq 0}$ and $(\zeta_{k})_{k \geq 0}$ coincide in law and hence the result holds also for $\lim_{k \to \infty} {\zeta_k[N]}/{k}$ and $\lim_{t\to \infty} \xi_t[N]/t$. 
The last limit holds since $(\zeta_k[N])_{k \geq 0}$ and $(\xi_{\tau_k}[N])_{k \geq 0}$ have the same distribution and the maximum displacements of $(\xi_t[N])_{t \geq 0}$ between jump times are tight (and hence converge to zero when divided by $t$). The existence of the limit $\lim_{t\to\infty}\xi_t[1]/t$ can be obtained in a similar way.

Finally observe that {\em (1)} and {\em (2)} of Theorem \ref{finiteN}, imply that $\etap{N-1}{t}$ converges in distribution, as $t \to \infty$ to an almost surely finite random variable $X$.
Therefore, ${\etap{N-1}{t}}/{t}=({\xi}_t[N]-{\xi}_t[1])/{t}$ converges in distribution to $0$ and thus, also in probability. This finishes the proof of {\em (3)}. 

\subsection{Velocity.} 
To prove the monotonicity stated in {\em (4)} we use a construction similar to the one in \eqref{coupling} to couple two systems, one with $N$ particles, denoted by $(\zeta^N_k)_{k \geq 0}$, and another with $N+1$ particles denoted by $(\zeta^{N+1}_k)_{k \geq 0}$. 
{This coupling is similar to the one used in \cite{BG, M, DR} to prove the same monotonicity result. 
The main difference is that in those models the rate at which each pair of particles produce a branching-selection event is independent of $N$, while in our case it is $1/(N-1)$}.

For the ease of the reading, we first describe the coupling informally. Let us suppose that the initial conditions are ordered in the following sense
\begin{equation}
\label{strong.order}
\zeta^{N+1}_0[i+1] \ge \zeta^{N}_0[i], \quad i=1, \ldots, N.
\end{equation}
We start by \emph{matching} the order statistics of both systems. Each $\zeta^N[i], i=1, \ldots, N$  is paired with $\zeta^{N+1}[i+1]$. {The particle at $\zeta^{N+1}[1]$} is not coupled.
We re-{match} the particles of the two processes just before and inmediately  after every time a jump takes place. 
We do not re-match between jumps, but during these periods the Brownian displacements are coupled among both systems according to the match. 
More precisely, after each jump we use the same Brownian motion to drive the particles at positions $\zeta^{N+1}[i+1]$ and $\zeta^{N}[i]$ , for every $1\le i \le N.$ 
This guarantees that the pairs matched by the coupling keep their relative order. 
In particular, each $i-$th order statistic of the $N-$system has at least $N-i$ particles of the $(N+1)-$system to the right.  

Given that the rate at which each  pair of order statistics produce a jump in a system with $N$ particles is $\tfrac{1}{N-1}$ (and hence depends on $N$), we couple the jumps as follows.
Each time there is a jump from position $\zeta^N[i]$ to position $\zeta^N[j]$ with $i<j,$ we enforce the same jump from $\zeta^{N+1}[i+1]$ to $\zeta^{N+1}[j+1]$ with probability $\tfrac{N-1}{N}$ while, with probability $\tfrac{1}{N}$, the jump occurs from $\zeta^{N+1}[1]$ to $\zeta^{N+1}[j+1]$.
Note that in order to obtain the correct rates for the $(N+1)-$system, extra jumps from $\zeta^{N+1}[1]$ to the other particles should be included.
Observe that {while this procedure gives the correct rates, neither of these jumps alter the ordering \eqref{strong.order}.

\bigskip
We now provide a precise construction. 
Let $(T_k)_{k \geq 1}$ a Poisson process with rate $\lambda^N+\lambda^+ $, with $\lambda^N:=\frac{N}{2}$ and $\lambda^+=1-\frac{1}{N-1}$.
Let $(V_k)_{k \geq 1}$, with $V \sim \mathcal{U}\{O_N\}$ as in \eqref{coupling}.
We decompose $B^{N+1}_s=(B^{1}_s,B^{N}_s)$ where $(B^{N}_s)_{s\geq0}$ is a $N$-dimensional Brownian motion. 
Finally, let $(W_k)_{k\geq 1}$ and $(C_k)_{k\geq 1}$ be two Bernoulli processes with success probability $\frac{\lambda^N}{\lambda^N + \lambda^+}$ and $\frac{N-1}{N}$ respectively, and $(U^+_k)_{k\geq 1}$ i.i.d. random variables with $U^+ \sim \mathcal{U}\{2,\ldots,N+1\}$. 
Given $\zeta^{N}_0,\zeta^{N+1}_0$ (and using the previously defined notations $\sigma$ and $\theta$), we construct both processes as follows:
$$
\zeta^{N}_{k+1}=
\begin{cases}
\sigma\big(\theta_{V_k}(\sigma(\zeta^{N}_{k} + {B}^{N}_{T_{k+1}}-{B}^N_{T_k} ))\big),  & \text{if $W_k=1$,}\\
\zeta^{N}_{k} + {B}^{N}_{T_{k+1}}-{B}^N_{T_k},  & \text{if $W_k=0$,}
\end{cases}
$$
$$
\zeta^{N+1}_{k+1}=
\begin{cases}
\sigma\big(\theta_{V_k+(1,1)}(\sigma(\zeta^{N+1}_{k} + {B}^{N+1}_{T_{k+1}}-{B}^{N+1}_{T_k} ))\big),  & \text{if $W_k=1, \ C_k=1$,}\\
\sigma\big(\theta_{(1,V_k(2)+1)}(\sigma(\zeta^{N+1}_{k} + {B}^{N+1}_{T_{k+1}}-{B}^{N+1}_{T_k} ))\big),  & \text{if $W_k=1, \ C_k=0$,}\\ 
\sigma\big(\theta_{(1,j)}(\sigma(\zeta^{N+1}_{k} + {B}^{N+1}_{T_{k+1}}-{B}^{N+1}_{T_k} ))\big),  & \text{if $W_k=0, \ U^+_k=j, \ j\ge 2$}.
\end{cases}
$$
It is easy to check that with this construction $$\zeta^{N+1}_k[i+1]>\zeta^{N}_k[i], \ \forall \ k\geq 1, \ \forall \ i=1, \ldots, N,$$ which implies the claimed monotonicity $v_{N+1} \ge v_N$.

\bigskip

To prove {\eqref{minimal}} we provide a lower bound for the speed of the empirical mean and an upper bound for the speed of the rightmost particle. The result will follow from the fact that both bounds coincide as $N\to \infty$. For the lower bound, we make use of the monotonicity of the spacings of $\xig{}{t}$, which is established in the next lemma.
\begin{lemma}
\label{mono}
Let $(\xi_t)_{t \geq 0}$, $(\tilde\xi_t)_{t \geq 0}$ be two processes with generator \eqref{generator} and initial conditions $\xig{}{0}$, $\tilde\xi_0$ respectively such that 
$$\osp{i+1}{0}-\osp{i}{0} \leq_{st} \tilde\xi_0[i+1]-\tilde\xi_0[i], \quad i=1, \ldots, N-1.$$ 
Then, 
$$\osp{i+1}{t}-\osp{i}{t} \leq_{st} \tilde\xi_t[i+1]-\tilde\xi_t[i], \quad i=1, \ldots, N-1, \quad \forall \ t \geq 0.$$ 
\end{lemma}
Here $\leq_{st}$ denotes stochastic domination.
\begin{proof}
Let us first observe that the result holds for the spacings of $N$ independent Brownian motions (see \cite[Theorem 4.1]{R2000}), which is the law of the particles between jump times. Next, note that if the jumps for both processes are coupled as in \eqref{coupling} the domination of spacings is preserved after each jump. Proceeding inductively we obtain the stochastic domination for all times.
\end{proof}

\noindent
{\emph{Lower bound.}} Let $u^0$ be the solution of the F-KPP equation with initial condition given by the Heavyside function and $w_{\sqrt{2}}(x)$ the minimal velocity traveling wave for \eqref{KPP}. Let $m(t)$ be the median of $u^0(\cdot,t)$. The following facts hold \cite{B83}.
\begin{enumerate}
\item[(i)] $\|u^0(\cdot + m(t),t) - w_{\sqrt{2}}\| \to 0$ as $t\to \infty$.

\item[(ii)] $\sqrt{2}= \int_{-\infty}^{+\infty} w_{\sqrt{2}}(x) (1-w_{\sqrt{2}}(x)) \ dx.$
\end{enumerate}
%
%
%
Our strategy consists in comparing the velocity of the empirical mean of the system in equilibrium as seen from the leftmost particle and the system with initial condition given by the Heavyside function. 
Afterwards, Theorem \ref{mresult} will provide us the link to use (i)-(ii). This idea appears in \cite{BCdMFLS}.

We denote by $\bar{\nu}^N$ the distribution obtained by placing particle labeled 1 at the origin and the remaining ones according to $\nu^N$. Let $m^N_t=\sum_{i=1}^N\xi_t(i)$ be the empirical mean and observe that 
\begin{equation*}
E_{\bar{\nu}^N}[m^N_t]=
\int_{0}^{+\infty} E_{\bar{\nu}^N}[1-F_N(x,t)] \ dx - \int_{-\infty}^{0} E_{\bar{\nu}^N}[F_N(x,t)] \ dx,
\end{equation*}
and hence differentiating under the integral sign we get
\begin{align}
\nonumber
\partial_t E_{\bar{\nu}^N}[m^N_t]&=
-\int_{-\infty}^{+\infty} \partial_t E_{\bar{\nu}^N}[F_N(x,t)] \ dx \\
\label{whatwewant}
&=\tfrac{N}{N-1} \int_{-\infty}^{+\infty}  E_{\bar{\nu}^N}[F_N(x,t)(1-F_N(x,t))] \ dx.
\end{align}
In the last equality we make use of \eqref{KPPwe}.
Let us notice that 
$$\int_{-\infty}^{+\infty}F_N(x,t)(1-F_N(x,t))dx=\sum_{k=1}^{N-1} \frac{k (N-k)}{N^2} (\osp{k+1}{t}-\osp{k}{t}),$$
and hence it is monotone in the spacings of the system.\\
We bound \eqref{whatwewant} from below as follows
\begin{align}
\label{linkKPP}
\nonumber \int_{-\infty}^{+\infty} E_{\bar{\nu}^N}[F_N(x,t)(1-F_N(x,t))] \ dx
& \geq
\int_{-\infty}^{+\infty} E_{0}[F_N(x,t)(1-F_N(x,t))] \ dx\\
& \geq
\int_{m(t)-b}^{m(t)+b} E_{0}[F_N(x,t)(1-F_N(x,t))] \ dx,
\end{align}
for all $0<b< \infty$.
Now, \eqref{linkKPP} can be written as
\begin{align}
\label{vd}
&\int_{m(t)-b}^{m(t)+b} E_{0}[F_N(x,t)(1-F_N(x,t))]\ dx
- \int_{m(t)-b}^{m(t)+b} u^0(1-u^0)(x,t) \ dx\\
\label{kpptf}
+ & \int_{m(t)-b}^{m(t)+b} u^0(1-u^0)(x,t) \ dx.
\end{align}
Due to (i)-(ii) we can fix $b$ and $t$ big enough to make \eqref{kpptf} as close to $\sqrt 2$ as desired. For those values of $b$ and $t$, the term \eqref{vd} converge to zero as $N\to \infty$. Here we use that, due to Theorem \ref{mresult}, $\|F_N(1-F_N)(\cdot,t) - u^0(1-u^0)(\cdot, t) \|$ converges to zero in probability (and hence in $L^1$) as $N\to\infty$. So,
\begin{equation}
\label{liminf.KPP}
\liminf_{N\to\infty} \int_{-\infty}^{+\infty} E_{\bar{\nu}^N}[F_N(x,t)(1-F_N(x,t))] \ dx \ge \sqrt{2}.
\end{equation}
We claim that $\partial_t E_{\bar{\nu}^N}[m^N_t]=v_N$.
Indeed, as $\nu^N$ is the equilibrium  distribution for the system seen from the leftmost particle, a renewal argument readily implies that $(\osp{1}{t})_{t \geq 0}$ has stationary increments and thus
\begin{equation*}
E_{\bar{\nu}^N}[m^N_t]=E_{\bar{\nu}^N}[m^N_t - \osp{1}{t}] + E_{\bar{\nu}^N}[\osp{1}{t}]=c_1+c_2 t,
\end{equation*}
with $c_1$ and $c_2$ depending on $N$.
The limit in \eqref{velocityE} forces $c_2$ to equal $v_N$ which implies our claim. Combining this with \eqref{whatwewant} and \eqref{liminf.KPP} we conclude that
$$\liminf_{n \to \infty} v_N\geq \sqrt{2}.$$
\noindent
\emph{Upper bound.}
For this bound, we embed the process into $N$ independent BBMs. Recall that the maximum $M_t$ of a BBM at time $t$ verifies 
\begin{equation}
\label{maximumBBM}
\lim_{t \to \infty} \tfrac{M_t}{t}=\sqrt{2} \quad  \text{a.s.}
\end{equation}


We provide a last alternative construction of the process.
Instead of using a Poisson clock of rate $1$ for the jumps of each particle, we use a Poisson clock of rate $\tfrac{k-1}{N-1}$ for each rank $k=1, \ldots,N$.
Whenever the $k$-th clock rings the particle whose rank is $k$ branches into two particles.
At that time, a particle is chosen uniformly at random among those particles with rank less than $k$ and is eliminated from the system. Between jumps, particles evolve as independent Brownian motions.
It is clear that (i) the process constructed in this way has generator \eqref{generator} and (ii) each particle branches at a rate which is less or equal than one. Hence we can couple this process with a BBM by adding extra branchings to achive rate one. 

More precisely, let us construct a coupling between a process with generator \eqref{generator} and $N$ independent BBM systems. We will still denote it $(\xi_t)_{t \geq 0}$. 
Let $X_1(t),\ldots,X_N(t), X_{N+1}(t), \ldots$ denote the positions of all the particles of $N$ independent BBM starting at $\xip{1}{0}, \ldots, \xip{N}{0}$, labeled according to its appearence time.
Let $(\tau_i)_{i \geq 1}$ be the branching times (in increasing order) and $(l_i)_{i \geq 1}$ the corresponding labels of the branching particles.
Let $(U_i)_{i \geq 1}$ and $(D_i^k)_{k \leq N,i \geq 1}$ be independent random variables with $U_i$ uniform in $[0,1]$, $i\ge 1$ and $D^k_i$ uniform in $\{1, \ldots , k-1\}$ for $i\ge 1$, $k\le N$.
Finally, let $\alpha_k=\tfrac{k-1}{N-1}, \ k=2, \ldots,N$. We construct a process $(L^1_t, \ldots, L^N_t)$ on the set of labels $\N$.
We will use these labels to determine the particles of the BBM that are going to be used to construct the coupling.
For $j=1, \ldots, N$, we define 
$$
R^j_t= \text{rank of $X_{L^j_t}(t)$ in the set $\{ X_{L^1_t}(t), \ldots, X_{L^N_t}(t) \}$} 
\subseteq
\{1, \ldots, N \}.
$$

We build $(L^1_t, \ldots, L^N_t)$ inductively as follows:
\begin{itemize}
\item
At time zero $(L^1_0, \ldots, L^N_0)=(1,\ldots,N)$.

\item
Assume $(L^1_{\tau_{i-1}}, \ldots, L^N_{\tau_{i-1}})$ is known.
For $t \in (\tau_{i-1}, \tau_{i}]$ we define:
\begin{equation*}
L^k_t=
\begin{cases}
N + i  & \text{ if $t=\tau_i$, \quad $l_i=L^j_t, \ j \neq  k$ , \quad $U_i < \alpha_{R^{j}_t}$, \quad $D^{R^{j}_t}_i= R^k_t$} \\
L^k_{t^-}  & \text{ otherwise,}
\end{cases}
\end{equation*} 
$k=1 , \ldots, N$.
\end{itemize}

We leave to the reader to verify that 
\begin{equation*}
(\xi_t(1), \dots, \xi_t(N)):=(X_{L^1_t}(t), \ldots, X_{L^N_t}(t)), \qquad t \geq 0
\end{equation*}
has generator \ref{generator}. With this coupling we have that 
\begin{equation}
\label{ENBBM}
\osp{N}{t} \leq \max(M^1_t, \ldots, M^N_t), \qquad \forall \ t \geq 0,
\end{equation}
with $M^j_t$ the maximum of the $j$-th BBM.
Combining \eqref{maximumBBM} with \eqref{ENBBM} we conclude 
$$\limsup_{t \to \infty} \tfrac{\osp{N}{t}}{t}\leq \sqrt{2}.$$

\bigskip

The proof is concluded by observing that
$$\osp{1}{t} \leq m^N_t \leq \osp{N}{t}, \qquad \forall \ t \geq 0.$$

\section{Acknowledgments}
We thank Nahuel Soprano-Loto for pointing out an error in a preliminary version of the manuscript. The authors are partially supported by grants UBACYT 20020160100147BA and PICT 2015-3154.

\bibliographystyle{plain}
\bibliography{biblio2}

\end{document}